\documentclass[11pt]{amsart}

\usepackage{mathrsfs} 
\usepackage{mathtools}
\usepackage{a4wide} 
\usepackage{graphicx} 
\usepackage{amssymb} 
\usepackage{epstopdf}
\usepackage{enumerate} 
\usepackage{titletoc} 
\usepackage[colorlinks=true, urlcolor=blue, linkcolor=blue, citecolor=black]{hyperref} %
\usepackage[nameinlink]{cleveref} 
\usepackage{esint} 
\usepackage{verbatim}
\usepackage{enumitem}

\numberwithin{equation}{section}

\theoremstyle{plain}
\newtheorem{theorem}{Theorem}[section]
\newtheorem{lemma}[theorem]{Lemma}
\newtheorem{corollary}[theorem]{Corollary}

\theoremstyle{definition}
\newtheorem{definition}[theorem]{Definition}
\newtheorem{remark}[theorem]{Remark}

\makeatletter
\@namedef{subjclassname@2020}{\textup{2020} Mathematics Subject Classification}
\makeatother

\title[Optimal regularity]{Optimal regularity for degenerate parabolic equations on a flat boundary}

\author{Hyungsung Yun}
\address{School of Mathematics, Korea Institute for Advanced Study, Seoul 02455, Republic of Korea}
\email{hyungsung@kias.re.kr}
  

\thanks{Hyungsung Yun has been supported by the KIAS Individual Grant (No. MG097801) at Korea Institute for Advanced Study.}

\begin{document}
\begin{abstract}
	We establish the optimal regularity of viscosity solutions to
	\begin{equation*}
		u_t - x_n^\gamma \Delta u = f,
	\end{equation*}
which arises in the regularity theory for the porous medium equation. Specifically, we prove that under the zero Dirichlet boundary condition on $\{x_n=0\}$, the optimal regularity of $u$ up to the flat boundary $\{x_n=0\}$ is $C^{1,1-\gamma}$. 
Moreover, for the homogeneous equations, we establish that the optimal regularity of $u$ is $C^{2,1-\gamma}$ in the spatial variables, and that $x_n^{-\gamma}u$ is smooth in the variables $x'$ and $t$.
\end{abstract}

\maketitle

%
%
\section{Introduction}
The main goal of this paper is to derive the optimal regularity of viscosity solutions to the following degenerate parabolic equations:
\begin{equation} \label{eq:main}
	u_t - x_n^\gamma \Delta u = f ,
\end{equation}
where $0<\gamma<1$ and $f$ is bounded and continuous.

Equation \eqref{eq:main} is closely related to the following Cauchy--Dirichlet problem for the porous medium equation:
\begin{equation} \label{eq:PME}
\left\{\begin{aligned}
	v_t - \Delta v^m &= 0&& \text{in } \Omega \times (0,\infty) \\
	v&=0 && \text{on } \partial\Omega \times  (0,\infty) \\
	v&=v_0 && \text{on } \Omega \times \{t=0\},
\end{aligned}\right.
\end{equation}
where $m>1$, $v_0>0$ in $\Omega$, and $\Omega \subset \mathbb{R}^n$ is a bounded domain with a smooth boundary. By applying appropriate transformations such as the hodograph transform, the change of variables that flattens $\partial \Omega$, and linearization, the problem \eqref{eq:PME} is transformed into the following Dirichlet problem:
\begin{equation}  \label{prob:zero}
\left\{\begin{aligned}
	u_t - x_n^\gamma \Delta u &= f && \text{in } Q_1^+ \coloneqq Q_1 \cap \{x_n>0\}\\
	u&=0 && \text{on } S_1 \coloneqq Q_1 \cap \{x_n=0\},
\end{aligned}\right.
\end{equation}
where $\gamma =1-1/m \in (0,1)$. The regularity results for the Dirichlet problem \eqref{prob:zero} can be applied to the Cauchy--Dirichlet problem \eqref{eq:PME} for the porous medium equation in bounded domains, we refer to \cite{KL09, KLY25}. For the unbounded domain, specifically $\Omega=\mathbb{R}^2$, we refer to \cite{DH98}. 
\subsection{Main results}
The global $C^{1,\alpha}$-regularity of viscosity solutions to \eqref{eq:main} was recently established by Lee and Yun \cite[Corollary 1.3]{LY25}. Here,  the exact value of $\alpha$ is unknown, but it cannot exceed $1-\gamma$.
To be precise, a direct computation shows that for any linear function $l(x)$,
\begin{equation*}	
	\varphi(x,t) = l(x) + 2t  + \frac{x_n^{2-\gamma}}{(2-\gamma)(1-\gamma)}
\end{equation*}
is a solution to  
\begin{equation*}
	u_t -x_n^\gamma \Delta u = 1.
\end{equation*}
Therefore, solutions to \eqref{eq:main} cannot be expected to be more regular than $C^{1,1-\gamma}$ up to the flat boundary $S_1$. The essence of our main results is that, under sufficiently smooth boundary data, the simple solution $\varphi$ actually represents a worst-case scenario in terms of boundary regularity.

The boundary regularity of solutions to \eqref{eq:main} is influenced by factors such as the geometry of the boundary and the Dirichlet boundary condition. Since the flat boundary $S_1$ is smooth, it does not present a significant obstacle to attaining optimal regularity of solutions to \eqref{eq:main}. However, the boundary regularity of solutions cannot exceed that of the boundary data, making the boundary condition a critical factor in determining the optimal regularity.

Our first main result is as follows. 
\begin{theorem} \label{thm:main}
	Suppose that $f \in C(Q_1^+) \cap L^\infty(Q_1^+)$ and $g \in C^{1,\beta}(\overline{S_1})$ for some $\beta \in (0,1)$. Let $u$ be a viscosity solution to 
\begin{equation}   \label{eq:prob}
\left\{\begin{aligned}
	u_t - x_n^\gamma \Delta u &= f && \text{in } Q_1^+ \\
	u&=g && \text{on } S_1.
\end{aligned}\right.
\end{equation}
Then $u \in C^{1,\bar\alpha}(\overline{Q_{1/2}^+})$ with the uniform estimate
\begin{equation*} 
	\|u\|_{C^{1,\bar\alpha}(\overline{Q_{1/2}^+})} \leq C (\|u\|_{L^{\infty}(Q_1^+)} + \|f\|_{L^{\infty}(Q_1^+)} + \|g\|_{C^{1,\beta}(\overline{S_1})}),
\end{equation*}
where $\bar\alpha= \min\{\beta, 1-\gamma\}$ and $C>0$ is a constant depending only on $n$, $\beta$ and $\gamma$.
\end{theorem}
\begin{remark}
If the Dirichlet boundary data $g$ is sufficiently smooth, for example, $g\in C^{1,1-\gamma}(\overline{S_1})$, then the optimal regularity of viscosity solutions to \eqref{eq:prob} is $C^{1,1-\gamma}$.
\end{remark}
\begin{corollary}
Suppose $f \in C(Q_1^+) \cap L^\infty(Q_1^+)$ and let $u$ be a viscosity solution to \eqref{prob:zero}. Then 
$u \in C^{1,1-\gamma}(\overline{Q_{1/2}^+})$ with the uniform estimate
\begin{equation*} 
	\|u\|_{C^{1,1-\gamma}(\overline{Q_{1/2}^+})} \leq C (\|u\|_{L^{\infty}(Q_1^+)} + \|f\|_{L^{\infty}(Q_1^+)}),
\end{equation*}
where $C>0$ is a constant depending only on $n$ and $\gamma$.
\end{corollary}

In the case of homogeneous equations, higher regularity can be expected. The global $C^{2,\alpha}$-regularity of viscosity solutions to 
\begin{equation} \label{eq:main_hom}
	u_t - x_n^\gamma \Delta u = 0 \quad \text{in } Q_1^+
\end{equation}
was recently proved by Lee and Yun \cite[Theorem 1.7]{LY25}. A direct calculation shows that, for any linear function $l(x)$,
\begin{equation*}	
	\psi(x,t) = l(x) + tx_n  + \frac{x_n^{3-\gamma}}{(3-\gamma)(2-\gamma)} 
\end{equation*}
is a solution to \eqref{eq:main_hom}. Therefore, solutions to \eqref{eq:main_hom} cannot be expected to possess more than $C^{2,1-\gamma}$ up to the flat boundary $\{x_n = 0\}$.
\begin{remark} \label{rmk_ut_bdd}
Let $u\in C^{2}(\overline{Q_1^+})$ be a solution to \eqref{eq:main_hom}  with boundary data $u=g$ on $S_1$. Note that
	\begin{equation*}
		|u_t(x,t)| \leq x_n^\gamma |\Delta u(x,t)| \leq \|u\|_{C^{2}(\overline{Q_1^+})} x_n^\gamma \quad \text{for all } (x,t) \in \overline{Q_1^+},
	\end{equation*}
	which implies that 
	\begin{equation} \label{cond_gt}
		g_t = 0 \quad \text{on } S_1.
	\end{equation}
Hence, condition \eqref{cond_gt} is a sufficient condition for the existence of a $C^2$-solution.
\end{remark}

Our second main result is as follows. 
\begin{theorem} \label{thm:main2}
	Suppose that $\beta  \in (0,1)$ and $g \in C^{2,\beta}(\overline{S_1})$ satisfies \eqref{cond_gt}. Let $u$ be a viscosity solution to 
\begin{equation}   \label{prob_xf}
\left\{\begin{aligned}
	u_t - x_n^\gamma \Delta u &= 0&& \text{in } Q_1^+ \\
	u&=g && \text{on } S_1.
\end{aligned}\right.
\end{equation}
Then
\begin{equation*}
	u(\cdot,t) \in C^{2,\bar\alpha}(\overline{B_{1/2}^+}) \quad \text{for all }t\in [-1/4,0] \qquad\text{and} \qquad
	x_n^{-\gamma} u_t \in C^{1-\gamma}(\overline{Q_{1/2}^+})
\end{equation*}
with the uniform estimates
\begin{equation*} 
	\|u(\cdot,t)\|_{C^{2,\bar\alpha}(\overline{B_{1/2}^+})} \leq C (\|u\|_{L^{\infty}(Q_1^+)} + \|g\|_{C^{2,\beta}(\overline{S_1})}) 
\end{equation*}
and 
\begin{equation*}
	\|x_n^{-\gamma}u_t\|_{C^{1-\gamma}(\overline{Q_{1/2}^+})} \leq C( \|u\|_{L^{\infty}(Q_1^+)} + \|g\|_{C^{2,\beta}(\overline{S_1})}) ,
\end{equation*}
where $\bar\alpha= \min\{\beta, 1-\gamma\}$ and $C>0$ is a constant depending only on $n$, $\beta$ and $\gamma$.
\end{theorem}

By setting $u=m^{1/\gamma}v^m$, the Cauchy--Dirichlet problems \eqref{eq:PME} can be reformulated as follows:\begin{equation}  \label{eq:PME_tr}
\left\{\begin{aligned}
	u_t - u^\gamma \Delta u &= 0&& \text{in } \Omega \times (0,\infty) \\
	u&=0 && \text{on } \partial\Omega \times(0,\infty)  \\
	u&=u_0 && \text{on } \Omega \times \{t=0\},
\end{aligned}\right.
\end{equation}
where $u_0 \coloneqq m^{1/\gamma} v_0^m$. The optimal regularity of solutions to \eqref{eq:PME_tr} was recently established as $C^{2,1-\gamma}$ by Jin, Ros-Oton, and Xiong \cite{JROX24} in the sense described below:
\begin{equation} \label{JROX24_op}
	\left\{\begin{aligned}
		u(\cdot, t) &\in C^{2,1-\gamma}(\overline\Omega) && \text{for all } t > T^*\\
		u(x, \cdot) &\in C^{\infty}(T^*,\infty) && \text{uniformly in } x \in \overline{\Omega},
	\end{aligned}\right.
\end{equation}
where the time $T^*>0$ satisfying
\begin{equation} \label{time_star} 
	u(\cdot,t) \in C^{0,1} (\overline{\Omega}) \quad \text{for all } t > T^*.
\end{equation}

When $\Omega=B_1^+$, the global Lipschitz regularity \eqref{time_star} can be interpreted as
\begin{equation*}
	u \approx x_n \quad \text{near } S_1.
\end{equation*}
Consequently, as in the \cite[Section 4]{Yun24}, the operator in \eqref{eq:main} can be understood as follows:
\begin{equation} \label{approx_op}
	u_t - x_n^\gamma \Delta u \approx u_t - u^\gamma \Delta u  \quad \text{near }S_1.
\end{equation}
Based on the approximation \eqref{approx_op}, it is reasonable to expect that the optimal regularity results in \eqref{JROX24_op} may also hold for \eqref{eq:main}.

The following corollary, in the case $\kappa=0$, is analogous to the optimal regularity results in \cite{JROX24}, and due to the structure of the equation, it indicates that all higher-order tangential derivatives, excluding those in the $x_n$-direction, are smooth.
\begin{corollary} \label{cor:main2}
Let $u$ be a viscosity solution to 
\begin{equation} \label{eq:f=g=0}
\left\{\begin{aligned}
	u_t - x_n^\gamma \Delta u &= 0 && \text{in } Q_1^+ \\
	u&=0 && \text{on } S_1.
\end{aligned}\right.
\end{equation}
Then
\begin{equation*}
	D_{(x',t)}^\kappa u(\cdot, t) \in C^{2,1-\gamma}(\overline{B_{1/2}^+}) \quad \text{for all $\kappa \in \mathbb{N}_0^n$ and $t \in [-1/4,0]$}
\end{equation*}
with the uniform estimate
\begin{equation*} 
	\|D_{(x',t)}^\kappa u(\cdot,t)\|_{C^{2,1-\gamma}(\overline{B_{1/2}^+})} \leq C \|u\|_{L^{\infty}(Q_1^+)},
\end{equation*}
where $C>0$ is a constant depending only on $n$ and $\gamma$.

Moreover,  
\begin{equation*}
	x_n^{-\gamma}D_{(x',t)}^\kappa u \in C^{1-\gamma}(\overline{Q_{1/2}^+}) \quad \text{for all $\kappa \in \mathbb{N}_0^n$}
\end{equation*}
with the uniform estimate
\begin{equation*}
	\|x_n^{-\gamma} D_{(x',t)}^\kappa u\|_{C^{1-\gamma}(\overline{Q_{1/2}^+})} \leq C \|u\|_{L^{\infty}(Q_1^+)} .
\end{equation*}
\end{corollary}
\subsection{Related works}
Equations exhibiting degeneracy on $S_1$ arise in various applied fields, including the porous medium equation, Gauss curvature flow, the Monge--Ampère equation, nonlocal equations, and financial mathematics.  For instance, the porous medium equation in $\mathbb{R}^2 \times (0,\infty)$ can be linearized as a degenerate parabolic equation
\begin{equation}\label{eq:DH98}
	u_t - x_2 (u_{x_1 x_1}+ u_{x_2 x_2})- \nu u_{x_2} =f 
\end{equation}
with $\nu>0$. Using the smoothness of solutions to \eqref{eq:DH98} in weighted H\"older space, Daskalopoulos and Hamilton \cite{DH98} established the smoothness of both the solution and the free boundary of the porous medium equation. 

For the porous medium equation in bounded domains, Kim and Lee \cite{KL09} established regularity results by applying $C_s^{2+\alpha}$-estimates in weighted H\"older spaces for solutions to \eqref{eq:main}. 
The results of \cite{DH98} suggest that solutions to the porous medium equation in bounded domains may also be smooth in weighted H\"older spaces. 
However, there are limitations to analyzing the higher regularity of solutions using these weighted H\"older spaces. This conjecture was completely resolved by Kim, Lee, and Yun \cite{KLY25}. They established regularity results for solutions to the porous medium equation in bounded domains by proving that solutions to \eqref{eq:main} admit infinite order fractional expansions by using the generalized Schauder theory \cite{KLY24} for fractional order expansions.

Research on nonlinear equations using the technique of linearizing them into degenerate equations has been actively pursued over the past two decades. For the Gauss curvature flow, see \cite{DH99, KLR13}; for the Monge--Ampère equation, see \cite{DS09, LS17}; for nonlocal equations, see \cite{CS07}; and for financial mathematics, see \cite{KLY23}.

Alongside the study of nonlinear equations, the regularity theory for equations with degeneracy or singularity on $S_1$ has advanced significantly over the past two decades.  Daskalopoulos and Lee \cite{DL03} established the Alexandroff--Bakelman--Pucci estimates, Harnack inequality, and the H\"older continuity of solutions to
\begin{equation*}
	u_t - a^{11} u_{x_1x_1} - x_2 a^{22} u_{x_2x_2} - 2 \sqrt{x_2}a^{12} u_{x_1 x_2}- b^1  u_{x_1} - b^2  u_{x_2} = f. 
\end{equation*}

Within the framework of regularity theory utilizing weighted H\"older spaces $C_s^{k+\alpha}$ associated with the Riemannian metric $ds$ introduced in \cite{DH98}, Song and Wang \cite{SW12} established $C_s^{2+\alpha}$-estimates of solutions to 
\begin{equation*}
	u_{x_1x_1} + x_2 a  u_{x_2x_2} -b u_{x_2}  = f. 
\end{equation*}
Additionally, Feehan and Pop \cite{FP13, FP14} proved $C_s^{2+\alpha}$-estimates of solutions to 
\begin{equation*}
	u_t - x_n a^{ij} D_{ij} u - b^i  D_i u - c u =f 
\end{equation*}
and a priori $C_s^{k+2+\alpha}$-estimates  of solutions to 
\begin{equation*}
	x_n a^{ij} D_{ij} u + b^i  D_i u - c u =f .
\end{equation*}

In the study of regularity theory based on the weighted Sobolev spaces, Dong and Phan \cite{DP21, DP23} established $W^{1,p}$-estimates of solutions to
\begin{equation*}
	x_n^\gamma(u_t + \lambda u) - D_i [x_n^\gamma (a^{ij} D_j u -F_i)] = \sqrt{\lambda} x_n^\gamma f.
\end{equation*}
Furthermore, Dong, Phan, and Tran \cite{DPT23, DPT24} established $W^{1,p}$-estimates of solutions to
\begin{equation*}
	u_t - x_n^\gamma D_i(a^{ij} D_j u) + \lambda c u =x_n^\gamma  D F + f 
\end{equation*}
and $W^{2,p}$-estimates of solutions to
\begin{equation*}
	u_t - x_n^\gamma \Delta u + \lambda u =x_n^\gamma  f  .
\end{equation*}
In the context of weighted Sobolev spaces with mixed norms, Dong and Ryu \cite{DR24} proved $W^{1,p,q}$-estimates of solutions to 
\begin{equation*}
	a_0u_t -x_n^2 D_i (a^{ij}D_j u) + x_n b^i D_i u + x_n D_i (\tilde{b}^i u) + cu +\lambda c_0u = D_i F_i + f.
\end{equation*}

Kim, Lee, and Yun \cite{KLY24} introduced a new framework for analyzing the regularity of solutions to 
\begin{equation*}
	u_t - \sum_{i,j=1}^{n-1} a^{ij} D_{ij} u - 2 x_n^{\gamma/2} \sum_{i=1}^{n-1} a^{in} D_{in}u - x_n^\gamma a^{nn} D_{nn}u - \sum_{i=1}^{n-1} b^i D_i u - x_n^{\gamma/2} b^n D_n u - cu =f
\end{equation*}
by employing fractional order expansions and developing a generalized Schauder theory to support this approach.

Over the past decade, active research on optimal regularity has employed the compactness argument. For the fully nonlinear equations, we refer to \cite{ART15, AS23}, and for the $p$-Laplace equations, we refer to \cite{ATU17, ATU18}. 

In the case of optimal regularity of solutions to \eqref{eq:PME_tr}, the approach differs from those previously mentioned. Specifically, Jin, Ros-Oton, and Xiong \cite{JROX24} proved that 
\begin{equation*}
	\frac{u}{\textnormal{dist}(\cdot,\partial \Omega)} \text{ is globally H\"older continuous} 
\end{equation*}
and by applying Schauder estimates, they obtained the optimal regularity results.

A key technique in this study is to interpret the parabolic equation as an elliptic one by leveraging the H\"older continuity of the time derivative. This approach can also be found in the study of the parabolic $p$-Laplace equation, we refer to \cite{LLYZ25} for details.

\subsection{Strategy of the proof}
The proof of the optimal regularity is divided into two main parts: the inhomogeneous case and the homogeneous case. Each part relies on a combination of compactness arguments and refined regularity theory.
\begin{enumerate} [label=(\roman*)]
\item Optimal regularity for the inhomogeneous equations: The key idea is to consider a distance function that preserves the parabolic scaling on $S_1$. This distance function enables us to capture the correct behavior of the solution along the normal direction. The strategy proceeds as follows:
\begin{itemize}
\item Introduce a distance function that preserves the parabolic scaling structure for \eqref{eq:main}.
\item We employ the boundary $C^{2,\alpha}$-regularity for the homogeneous equations from \cite{LY25} as a foundational estimate..
\item We exploit interior $C^{1,1^-}$-regularity to obtain control away from the boundary.
\item A compactness argument, combined with boundary blow-up argument, leads to the optimal regularity up to the boundary.
\end{itemize}
\item Optimal regularity for the homogeneous equations: This part of the analysis is based on differentiating the equation in time and exploiting the special structure of $u_t$ near the boundary:
\begin{itemize}
\item Noting that $u_t \approx x_n$ near $S_1$, we analyze the weighted time derivative $x_n^{-\gamma} u_t$, which turns out to be H\"older continuous.
\item This structure allows us to reduce the parabolic equation to an elliptic one in suitable coordinates, and we apply classical elliptic regularity theory to conclude.
\item When the boundary data vanishes, we apply a bootstrap argument to show that all higher-order tangential derivatives (excluding the $x_n$-direction) satisfy optimal regularity estimates.
\end{itemize}
\end{enumerate}

\subsection{Organization of the paper}
The paper is organized as follows.
\begin{itemize}
\item \Cref{sec:pre} provides preliminary material necessary for our analysis. We introduce notations, define the appropriate Hölder spaces, and recall the notion of viscosity solutions. We also present known regularity results that serve as key tools throughout the paper.
\item \Cref{sec:OR} is devoted to establishing the main results on optimal regularity. We divide the analysis into two parts:
\begin{itemize}
\item In \Cref{sec:OR_inhom}, we prove the optimal $C^{1,1-\gamma}$-regularity of viscosity solutions to the inhomogeneous equation up to the flat boundary under regularity assumptions on the data. The argument combines compactness methods with a boundary blow-up argument.
 \item In \Cref{sec:OR_hom}, we address the homogeneous case. Here, we exploit the structure of the equation to derive sharp time derivative estimates and use them to reduce the problem to elliptic equations on time slices. This leads to the optimal $C^{2,1-\gamma}$-regularity of the tangential derivatives.
\end{itemize}
\end{itemize}
%
%
\section{Preliminaries} \label{sec:pre}
This section provides a summary of notations, along with the definitions and established regularity results that are used throughout the paper.
\subsection{Notations} 
The following are the notations used in this paper.
\begin{enumerate} [label=(\arabic*)]
	\item Sets: 
		\begin{itemize}
			\item $\mathbb{R}^n_+ \coloneqq \{ x\in \mathbb{R}^n : x_n>0\}$.
			\item $\mathbb{N}_0 \coloneqq \mathbb{N} \cup \{0\}$.
			\item $B_r(x_0)\coloneqq \{x\in \mathbb{R}^n : |x-x_0| < r\} $.
			\item $B_r^+(x_0)\coloneqq B_r(x_0)  \cap \mathbb{R}^n_+$.
			\item $Q_r(x_0,t_0) \coloneqq B_r(x_0) \times ( t_0-r^{2-\gamma} , t_0]$.
			\item $Q_r^+(x_0,t_0)\coloneqq  B_r^+(x_0) \times ( t_0-r^{2-\gamma} , t_0]$.
			\item $Q_r\coloneqq Q_r(0,0)$ and $Q_r^+\coloneqq Q_r^+(0,0)$.
			\item $\mathcal{Q}_r(x_0,t_0) \coloneqq B_r(x_0) \times (t_0-r^2,t_0]$.
			\item $S_r \coloneqq \{(x,t) \in Q_r : x_n =0 \}$.
		\end{itemize}
	\item Points:
		\begin{itemize}
			\item $x'\coloneqq(x_1,\cdots, x_{n-1})$ for $x=(x_1,\cdots,x_n)$.
			\item $e_n \coloneqq (0',1) \in \mathbb{R}^n$.
		\end{itemize}
	\item Multiindex: A vector $\kappa =(\kappa_1,\cdots,\kappa_k) \in \mathbb{N}_{0}^k$ is called a $k$-dimensional multiindex of order $|\kappa| = \kappa_1 + \cdots + \kappa_k$.
	\item Let $\kappa =(\kappa_1,\cdots,\kappa_{n}) \in \mathbb{N}_0^n$ be a multiindex. We define the $|\kappa|$-th order partial derivatives of $u$ in the $(x',t)$-directions as follows:
	\begin{equation*}
	D_{(x',t)}^{\kappa} u \coloneqq \frac{\partial^{|\kappa|} u}{\partial x_1^{\kappa_1} \cdots \partial x_{n-1}^{\kappa_{n-1}}  \partial t^{\kappa_n}} .
\end{equation*}

	\item Let $\kappa=(\kappa_1,\cdots,\kappa_{n+1}) \in \mathbb{N}_0^{n+1}$ be a multiindex. The norm of a polynomial of 
\begin{equation*}
	P(x,t)=\sum_{\kappa_1+ \cdots + \kappa_n+2\kappa_{n+1} \leq k} \frac{ c_{\kappa} x_1^{\kappa_1} \cdots x_n^{\kappa_n} t^{\kappa_{n+1}} }{\kappa_1!   \cdots \kappa_{n+1}!} 
	\quad \text{with $\deg P=k$}
\end{equation*}	
is defined as follows:
\begin{equation*}
	\|P\| =  \sum_{\kappa_1+ \cdots + \kappa_n+2\kappa_{n+1} \leq k} |c_{\kappa}|.
\end{equation*}

\end{enumerate}
\subsection{H\"older spaces}
We define the H\"older spaces used throughout the paper.
\begin{definition} [H\"older spaces]
Let $\Omega \subset \mathbb{R}^{n+1}$ be a domain and let $\alpha \in (0,1)$.
\begin{enumerate} [label=(\roman*)]
	\item The space $C^{\alpha}(\overline{\Omega})$ consists of all functions $u \in C(\overline{\Omega} )$ satisfying the following condition holds: there exists $C>0$ such that 
	\begin{equation*}
		|u(x,t)-u(y,s)| \leq C(|x-y|+\sqrt{|t-s|})^{\alpha}
	\end{equation*}
	for all $(x,t), (y,s) \in \Omega$.
	\item The space $u \in C^{1,\alpha}(\overline{\Omega})$  consists of all functions $u\in C(\overline{\Omega})$ satisfying the following condition holds: for each $(x_0,t_0) \in \overline{\Omega}$, there exist $C>0$, $r>0$ and a linear function $L$ such that
	\begin{equation*}
		|u(x,t)-L(x)|\leq C(|x-x_0|+\sqrt{|t-t_0|})^{1+\alpha} 
	\end{equation*}
	for all $(x,t) \in \Omega \cap Q_r(x_0,t_0)$.
	\item The space $ C^{2,\alpha}(\overline{\Omega})$ consists of all functions $u\in C(\overline{\Omega})$ satisfying the following condition holds: for each $(x_0,t_0) \in \overline{\Omega}$, there exist $C>0$, $r>0$ and a quadratic polynomial $P$ such that
	\begin{equation*}
		|u(x,t)-P(x,t)|\leq C(|x-x_0|+\sqrt{|t-t_0|})^{2+\alpha} 
	\end{equation*}
	for all   $(x,t) \in \Omega \cap Q_r(x_0,t_0)$.
\end{enumerate}
\end{definition} 
Unlike interior H\"older continuity, H\"older continuity at the boundary points requires considering a metric that is invariant under the scaling of  \eqref{eq:main}, we refer to \cite[Section 2.1]{LY25}.
\begin{definition} [Pointwise H\"older continuity at the flat boundary] \label{bd_holder}
Let $(x_0,t_0) \in S_1$ and let $k\in \mathbb{N}$.
\begin{enumerate} [label=(\roman*)]
	\item We say that $g$ is $C^{k,\alpha}$ at $(x_0,t_0)$, denoted by $g \in C^{k,\alpha}(x_0,t_0)$, if there exist constant $C>0$, $r>0$ and a polynomial $P(x,t)$ with $\deg P \leq k$ such that 
	\begin{equation*} 
		|g(x,t)-P(x,t)| \leq C(|x-x_0|+|t-t_0|^{\frac{1}{2-\gamma}})^{k+\alpha} 
	\end{equation*}
	for all  $(x,t) \in Q_1^+ \cap Q_r^+(x_0,t_0)$.
	\item The space $ C^{k,\alpha}(\overline{S_1})$ consists of all functions $g \in C(\overline{S_1})$ such that
	\begin{equation*}
		g\in C^{k,\alpha}(x_0,t_0) \quad \text{for all }  (x_0,t_0) \in \overline{S_1}.
	\end{equation*}
\end{enumerate}
\end{definition}
\begin{remark}
Since $0<\gamma<1$, the metric 
\begin{equation*}
	d[(x,t),(y,s)] \coloneqq |x-y|+|t-s|^{\frac{1}{2-\gamma}}
\end{equation*}
in \Cref{bd_holder} has a higher-order in the time distance compared to the standard parabolic metric. Thus, the global regularity is ultimately governed by the regularity with respect to the standard parabolic metric.
\end{remark}
\subsection{Viscosity solutions}
We introduce the concept of solutions to \eqref{eq:main}. Since the main equation \eqref{eq:main} has a nondivergence structure, we adopt the notion of viscosity solutions.
\begin{definition}[Viscosity solutions] 
Suppose that $f$ is a continuous function in $Q_1^+$.
\begin{enumerate} [label=(\roman*)]
	\item Let $u$ be a upper semicontinuous function in $Q_1^+$. We say that $u$ is a \textit{viscosity subsoution} of \eqref{eq:main} when the following condition holds: if for each $(x,t) \in Q_1^+ $ and each function $\varphi \in C^2(Q_1^+)$, $u-\varphi$ attains a local maximum at $(x, t)$,  then
		\begin{equation*}
			\varphi_t (x,t) - x_n^\gamma \Delta \varphi (x,t )  \leq f(x,t).
		\end{equation*}
	\item Let $u$ be a lower semicontinuous function in $Q_1^+$. We say that $u$ is a \textit{viscosity supersoution} of \eqref{eq:main} when the following condition holds:  if for each $(x,t) \in Q_1^+ $ and each function $\varphi \in C^2(Q_1^+)$, $u-\varphi$ attains a local minimum at $(x, t)$,  then
		\begin{equation*}
			\varphi_t (x,t) - x_n^\gamma \Delta \varphi (x,t )  \ge f(x,t).
		\end{equation*}
\end{enumerate} 
\end{definition}
The existence and uniqueness of viscosity solutions to \eqref{eq:main} are obtained via the comparison principle and Perron method, as proved in \cite{LY25}.
\begin{theorem} [Existence of the viscosity solutions, {\cite[Lemma 2.16]{LY25}}]
Suppose that  $f \in C(Q_1^+) \cap L^{\infty}(Q_1^+)$ and $g \in C(\partial_p Q_1^+)$.
Then the Cauchy--Dirichlet problem
\begin{equation*}  
	\left\{\begin{aligned}
		u_t - x_n^{\gamma} \Delta u &= f && \text{in } Q_1^+\\
		u &= g && \text{on } \partial_p Q_1^+
	\end{aligned} \right.
\end{equation*}
has a unique viscosity solution.
\end{theorem}
\subsection{Known Regularity Results}
The boundary Lipschitz estimates of solutions to \eqref{eq:main} was proved in \cite{LY25}. The following lemma is needed to obtain time derivative estimates for the homogeneous equations.
\begin{lemma} [Boundary Lipschitz estimates, {\cite[Lemma 3.1 and Remark 3.2]{LY25}}]\label{thm:LY25_lip}
Suppose $f\in C(Q_1^+) \cap L^\infty(Q_1^+)$. Let $u$ be a viscosity solution to 
\begin{equation*}  
\left\{\begin{aligned}
	u_t - x_n^\gamma \Delta u &= f && \text{in } Q_1^+ \\
	u&=0 && \text{on }S_1.
\end{aligned}\right.
\end{equation*}
Then 
\begin{equation*}
	|u(x,t)|\leq C (\|u\|_{L^\infty(Q_1^+)}  + \|f\|_{L^\infty(Q_1^+)})x_n \quad \text{for all } (x,t) \in \overline{Q_{1/2}^+},
\end{equation*}
where $C>0$ is a constant depending only on $n$ and $\gamma$.
\end{lemma}

The $C^{2,\alpha}$-regularity of viscosity solutions to \eqref{eq:main_hom} was proved in \cite{LY25}. By considering  $F(M)=\textnormal{tr}(M)$ in \cite[Lemma 4.1]{LY25}, we obtain the following boundary $C^{2,\alpha}$-estimates.
\begin{lemma} [Boundary $C^{2,\alpha}$-estimates, {\cite[Lemma 4.1]{LY25}}]\label{thm:LY25_c2a}
Let  $u$ be a viscosity solution to \eqref{eq:f=g=0}. Then, there exist a constant $\sigma=\sigma(n,\gamma) \in (0, 1)$ and  a polynomial 
\begin{equation*}
	P(x)= a_0x_n + \sum_{i=1}^n a_i x_i x_n
\end{equation*}
such that $\Delta P =0$, $\|P\| \leq C_1$ and
\begin{equation*}
	|u(x,t)-P(x)|\leq C_1\|u\|_{L^\infty(Q_1^+)} |x|^{1+\sigma}x_n \quad \text{for all } (x,t) \in \overline{Q_{1/2}^+},\end{equation*}
where $C_1>0$ is a constant depending only on $n$, $\sigma$ and $\gamma$. 

Moerover, there exists $a\in \mathbb{R}$  such that $|a|\leq C_1$ and  
\begin{equation*}
	|u_t(x,t)-ax_n| \leq C_1\|u\|_{L^\infty(Q_1^+)} (|x|+|t|^{\frac{1}{2-\gamma}})^{\sigma}  x_n \quad \text{for all } (x,t) \in \overline{Q_{1/2}^+}.
\end{equation*}
\end{lemma}

To prove the main lemmas using a compactness argument, the following stability theorem is required.
\begin{lemma}[Stability theorem, {\cite[Theorem 2.10]{LLY24}}]\label{thm-stability}
	Suppose that $f^k \in C(Q_1^+) \cap L^\infty(Q_1^+)$ for all $k \in \mathbb{N}$. Let $u^k$ be a viscosity subsolution (resp. supersolution) of
	\begin{align*}
		\partial_t u^k - x_n^\gamma \Delta u^k  = f^k \quad \text{in $Q_1^+$}.
	\end{align*}
Moreover, assume that 
\begin{equation*}
	\text{$u^k \to u$ locally uniformly} \quad \text{and} \quad \text{$f^k \to f$ locally uniformly}
\end{equation*}
as $k \to \infty$. Then $u$ is a viscosity subsolution (resp. supersolution) of
\begin{align*}
	u_t -x_n^\gamma \Delta u = f \quad \text{in $Q_1^+$}.
\end{align*}
\end{lemma}

%
%
\section{Optimal regularity} \label{sec:OR}

Throughout this section, the constants $\bar\alpha$,  $\sigma$ and $C_1$ refer to the constants from \Cref{thm:main}, \Cref{thm:main2} and \Cref{thm:LY25_c2a}, respectively.

\subsection{Optimal regularity for inhomogeneous equations} \label{sec:OR_inhom}
In this subsection, we address the optimal regularity of solutions to \eqref{eq:prob}. Although the exponent of the constant $\eta$ appearing in  \Cref{lem:first_step} and \Cref{lem:next_step} cannot be precisely determined in the general case, we will choose the exponent of $\eta$ in these lemmas to be optimal.
\begin{lemma} \label{lem:first_step}
There exists $\delta >0$ depending only on $n$  and $\gamma$ such that if $u$ is a viscosity solution to \eqref{eq:prob} satisfying
\begin{equation*}
	 \|u\|_{L^{\infty}(Q_1^+)} \leq 1, \quad 
	 \|f\|_{L^{\infty}(Q_1^+)} \leq \delta \quad \text{and} \quad
	 \|g\|_{L^{\infty}(S_1)} \leq \delta, 
\end{equation*}
then there exists a constant $a \in \mathbb{R}$ such that
\begin{equation*}
	\|u-ax_n\|_{L^{\infty}(Q_{\eta}^+)} \leq \eta^{2-\gamma}  \quad \text{and} 
	\quad |a| \leq C_1,
\end{equation*}
where $\eta \in (0,1)$ is a constant depending only on $n$ and $\gamma$.
\end{lemma}

\begin{proof}
Let us assume that the conclusion is false, that is, there exists a sequence of functions $\{(u^k,f^k,g^k) \mid k \in \mathbb{N}\}$ such that  
\begin{equation*}
	\|u^k\|_{L^{\infty}(Q_1^+)} \leq 1, \quad \|f^k\|_{L^{\infty}(Q_1^+)} \leq \frac{1}{k}, \quad   \|g^k\|_{L^{\infty}(S_1)} \leq \frac{1}{k}
\end{equation*}
and
\begin{equation*}
	\left\{\begin{aligned}
		\partial_t u^k -x_n^\gamma \Delta u^k &= f^k  && \text{in } Q_1^+ \\
		u^k &= g^k && \text{on } S_1.
	\end{aligned}\right.
\end{equation*}
Moreover, the following estimates hold:
\begin{equation}\label{eq:contradiction}
	\|u^k-ax_n\|_{L^{\infty}(Q_{\eta}^+)} >\eta^{2-\gamma} \quad \text{for all }  a \in \mathbb{R} \text{ with } |a| \leq C_1,
\end{equation}
where $\eta \in (0,1)$ will be determined later.

Since $\|u^k\|_{L^{\infty}(Q_1^+)} \leq 1$ for all $k \in \mathbb{N}$, it follows from the interior H\"older regularity for uniformly parabolic equations that $\{u^k \mid  k\in \mathbb{N} \}$ is equicontinuous in any compact set $K \subset  Q_1^+$. So, by Arzela-Ascoli theorem, there exist a subsequence $\{u^{k_j} \}$ of $\{u^k \mid  k\in \mathbb{N} \}$ and a function $u^\infty \in C(Q_1^+)$ such that 
\begin{equation*}
	\text{$u^{k_j} \to u^\infty$ uniformly in any compact set $K \subset Q_1^+$ as $j \to \infty$.}
\end{equation*}
Since $\|f^{k_j}\|_{L^{\infty}(Q_1^+)} \leq 1/k_j$ and $\|g^{k_j}\|_{L^{\infty}(S_1)} \leq 1/k_j$, we can see that
\begin{equation*}
	\lim_{j \to \infty} \|f^{k_j}\|_{L^{\infty}(Q_1^+)} = 0 \quad \text{and} \quad 
	\lim_{j \to \infty} \|g^{k_j}\|_{L^{\infty}(S_1)} = 0.
\end{equation*}
So,  \Cref{thm-stability} yields  that 
\begin{equation*}
	\left\{\begin{aligned}
		\partial_t u^\infty -x_n^\gamma \Delta u^\infty & =0 && \text{in } Q_1^+ \\
		u^\infty &= 0 && \text{on } S_1.
	\end{aligned}\right.
\end{equation*}
Then, by \Cref{thm:LY25_c2a}, there exists $P(x)=a_0 x_n + \sum_{i=1}^na_i x_i x_n$ such that
\begin{align*}
	|u^\infty (x,t)- a_0 x_n| \leq \left(\sum_{i=1}^n |a_i x_i| + C_1|x|^{1+\sigma} \right)x_n \quad \text{for all $(x,t) \in \overline{Q_{1/2}^+}$}
\end{align*}
and $a_0 \leq \|P\| \leq C_1$.

Finally, by taking $\eta < (\frac{1}{4C_1})^{1/\gamma}$,  we have
\begin{equation} \label{eq:contradiction2}
	\|u^\infty-a_0 x_n\|_{L^{\infty}(Q_{\eta}^+)} \leq (\|P\| \eta + C_1 \eta^{1+\sigma}) \eta \\
	\leq 2C_1 \eta^2 \\
	< \frac{1}{2} \eta^{2-\gamma}.
\end{equation}

On the other hand, letting $k=k_j \to \infty$ in \eqref{eq:contradiction} for $a=a_0$ yields 
\begin{align*}
	\|u^\infty-a_0 x_n\|_{L^{\infty}(Q_{\eta}^+)} \ge\eta^{2-\gamma},
\end{align*}
which contradicts \eqref{eq:contradiction2}.
\end{proof}
\begin{lemma} \label{lem:next_step}
	Let  $\delta>0$ and $\eta >0$ be as in \Cref{lem:first_step}. Let $\beta \in(0,1)$ and let  $u$ be a viscosity solution to \eqref{eq:prob} satisfying 
\begin{equation*}
	\|u\|_{L^{\infty}(Q_1^+)} \leq 1, \quad \|f\|_{L^{\infty}(Q_1^+)} \leq \delta
\end{equation*}
and 
\begin{equation} \label{c1b_g}
	|g(x,t)| \leq 2^{-1-\beta} \delta (|x|+|t|^{\frac{1}{2-\gamma}})^{1+\beta} \quad \text{for all } (x,t) \in S_1.
\end{equation}
Then there exists a sequence $\{a_k \mid k\ge -1 \}$
such that for each $k \geq 0$, we have 
\begin{align}\label{induction}
	|a_k-a_{k-1}| \leq  C_1 \eta^{(k-1) \bar\alpha}
	\quad \text{and} \quad  
	\|u-a_kx_n\|_{L^{\infty}(Q_{\eta^k}^+)} \leq \eta^{k(1+\bar\alpha)} .
\end{align}
\end{lemma}

\begin{proof}
	We prove it by induction. 
	\begin{enumerate} [label=(\roman*)]
	\item In the case $k=0$, setting  $a_{-1}=a_0=0$ makes \eqref{induction} hold.
	\item Assume that \eqref{induction} holds for $k\ge 0$. Then we will prove that  \eqref{induction} also holds for $k+1$.
		Let $r=\eta^k$, $y=x/r$, $s=t/r^{2-\gamma}$ and consider the following function 
	\begin{align*}
		v(y,s) \coloneqq \frac{u(x,t)-a_kx_n}{r^{1+\bar\alpha}}.
	\end{align*}
	Then $v$ is a viscosity solution to
	\begin{equation*}
		\left\{\begin{aligned}
			\partial_s v - x_n^\gamma \Delta v&= \widetilde{f}  && \text{in } Q_1^+ \\
			v &= \widetilde{g} && \text{on }  \{y_n=0\},
		\end{aligned}\right.
	\end{equation*}
where
\begin{equation*}
	\widetilde{f}(y,s)\coloneqq r^{1-\gamma -\bar\alpha} f(x,t)  \quad \text{and} \quad  \widetilde{g}(y,s)\coloneqq \frac{g(x,t)-a_kx_n}{r^{1+\bar\alpha}}.
\end{equation*}
Then $\|\widetilde{f}\|_{L^{\infty}(Q_1^+)} = r^{1-\gamma -\bar\alpha} \|f\|_{L^{\infty}(Q_r^+)}\leq \delta$ and by the induction hypothesis, we obtain $\|v\|_{L^{\infty}(Q_1^+)} \leq 1$. Moreover, from \eqref{c1b_g}, we have 
\begin{equation*}
	\hskip 1.2cm |\widetilde{g}(y,s)|=\frac{|g(ry,r^{2-\gamma}s)|}{r^{1+\bar\alpha}} \leq \frac{2^{-1-\beta}}{r^{1+\bar\alpha}}\cdot \delta \cdot (2r)^{1+\beta} \leq \delta \quad \text{for all } (y,s) \in \{y_n=0\}.
\end{equation*}
Therefore, since $v$ satisfies the assumptions of \Cref{lem:first_step}, there exists a constant $\tilde{a}\in\mathbb{R}$ such that 
\begin{equation}\label{eq:induc2_bd}
	|\tilde{a}| \leq C_1
\end{equation}
and
\begin{equation}\label{eq:induc1_bd}
	\|v-\tilde{a} y_n\|_{L^{\infty}(Q_{\eta}^+)} \leq \eta^{2-\gamma} \leq \eta^{1+\bar\alpha}.
\end{equation}
Setting $a_{k+1} = a_k + r^{\bar\alpha} \tilde{a}$, it follows from \eqref{eq:induc2_bd} that
\begin{align*}
	|a_{k+1}-a_k| =r^{\bar\alpha} |\tilde{a}| \leq  C_1 \eta^{k\bar\alpha}.
\end{align*}
Furthermore, by  \eqref{eq:induc1_bd}, we obtain 
\begin{align*}
	|u(x,t)-a_{k+1}x_n| &=|u(x,t)-a_kx_n-r^{\bar\alpha}\tilde{a}x_n | \\
	&\leq r^{1+\bar\alpha} \left|v\left(\frac{x}{r}, \frac{t}{r^{2-\gamma}}\right)-\tilde{a} \cdot \frac{x_n}{r}\right| \\
	&\leq \eta^{(k+1)(1+\bar\alpha)}
\end{align*}
for all $(x,t) \in  Q_{\eta^{k+1}}^+$, so we conclude that \eqref{induction} also holds for $k+1$.
	\end{enumerate}
Therefore, by induction, \eqref{induction}  holds for all $k\ge 0$.
\end{proof}
From \Cref{lem:first_step} and \Cref{lem:next_step}, together with an appropriate reduction scheme, we obtain the following theorem on optimal regularity.
\begin{theorem}[Optimal boundary $C^{1, \bar\alpha}$-estimates]  \label{thm:point_c1a}
	Let $(x_0,t_0) \in S_{1/2}$ and $\beta \in (0,1)$. Assume that $u$ is a viscosity solution to \eqref{eq:prob} satisfying $f\in C(Q_1^+) \cap L^{\infty}(Q_1^+)$ and $g\in C^{1,\beta}(x_0,t_0)$. Then $u \in C^{1, \bar\alpha}(x_0,t_0)$ and there exists a linear function $L$ such that
	\begin{equation*}
		|u(x,t)-L(x)|\leq CN(|x-x_0|+|t-t_0|^{\frac{1}{2-\gamma}})^{1+\bar\alpha} \quad \text{for all }(x,t) \in Q_1^+ \cap Q_1^+(x_0,t_0)
	\end{equation*}
	and 
	\begin{equation*}
		|Du(x_0,t_0)| \leq CN,
	\end{equation*}
	where 
	\begin{equation*}
		N\coloneqq \|u\|_{L^{\infty}(Q_1^+)}+\|f\|_{L^{\infty}(Q_1^+)}+ \|g\|_{C^{1,\beta}(x_0,t_0)}
	\end{equation*}
	and $C>0$ is a constant depending only on $n$, $\beta$ and $\gamma$.
\end{theorem}

\begin{proof}
Without loss of generality, we may assume that $(x_0,t_0)$ is the origin. First, let us show that the assumptions of \Cref{lem:next_step} hold through an appropriate reduction argument.

Since $g\in C^{1,\beta}(0,0)$, there exists a linear function $L_0(x)$ such that the following estimate holds:
\begin{equation*}
	|g_0(x,t)| \leq \|g\|_{C^{1,\beta}(0,0)} (|x|+|t|^\frac{1}{2-\gamma})^{1+\beta} \quad \text{for all }(x,t) \in S_1,
\end{equation*}
where $g_0 \coloneqq g-L_0$. Let $v(x,t)=u(x,t)-L_0(x)$, then $v$ is a viscosity solution to 
\begin{equation*} 
	\left\{\begin{aligned}
		v_t - x_n^\gamma \Delta v &= f  && \text{in } Q_1^+  \\
		v&= g_0 && \text{on }S_1.
	\end{aligned}\right.
\end{equation*}	
For 
\begin{equation*}
	N_\delta \coloneqq  \|v\|_{L^{\infty}(Q_1^+)}+ \delta^{-1}\|f\|_{L^{\infty}(Q_1^+)} + 2^{1+\beta}\delta^{-1}\|g\|_{C^{1,\beta}(0,0)},
\end{equation*}
let us consider the following functions:
\begin{equation*}
	\tilde{u}(x,t) \coloneqq \frac{v(x,t)}{N_\delta}, \quad 
	\tilde{f}(x,t) \coloneqq \frac{f(x,t)}{N_\delta} \quad \text{and} \quad 
	\tilde{g}(x,t) \coloneqq \frac{g_0(x,t)}{N_\delta}.
\end{equation*}
Then we have 
\begin{equation*}
	\|\tilde{u}\|_{L^{\infty}(Q_1^+)} \leq 1, \quad 
	\|\tilde{f}\|_{L^{\infty}(Q_1^+)} \leq \delta  
\end{equation*}
and 
\begin{equation*} 
	|\tilde{g}(x,t)| \leq 2^{-1-\beta} \delta  (|x|+|t|^{\frac{1}{2-\gamma}})^{1+\beta} \quad \text{for all }  (x,t) \in S_1.
\end{equation*}
That is, $\tilde{u}$, $\tilde{f}$ and $\tilde{g}$ satisfy the assumptions of \Cref{lem:next_step}, so there exists $\{a_k \mid k\ge -1\}$ satisfying \eqref{induction} for $\tilde{u}$. From the first recurrence relation in \eqref{induction}, the following inequalities hold for $i \ge j$: 
\begin{align*}
	|a_i-a_j| &\leq |a_i - a_{i-1}| + \cdots + |a_{j+1}-a_j| \\
	&\leq C_1 \frac{\eta^{j\bar\alpha}(1-\eta^{(i-j)\bar\alpha})}{1-\eta^{\bar\alpha}} \\
	&\leq \frac{C_1}{1-\eta^{\bar\alpha}} \eta^{j\bar\alpha}.
\end{align*}
Thus, $\{a_k \mid k\ge -1\}$ is a Cauchy sequence, and its limit $a_\infty =\lim_{k\to \infty} a_k$ exists satisfying
\begin{equation*}
	|a_\infty-a_j| \leq \frac{C_1}{1-\eta^{\bar\alpha}} \eta^{j\bar\alpha}.
\end{equation*}

For each $(x,t) \in Q_1^+$, choosing $j \ge 0$ such that  $\eta^{j+1} \leq \max\{ |x|, |t|^{\frac{1}{2-\gamma}} \}< \eta^j$ holds, we obtain 
\begin{align*}
	|\tilde{u}(x,t)-a_{\infty}x_n| &\leq |u(x,t)-a_j x_n|+|a_j -a_{\infty}| x_n \\
	&< \eta^{j(1+\bar\alpha)} + \frac{C_1}{1-\eta^{\bar\alpha}} \eta^{j (1+\bar\alpha)}  \\
	&=\frac{1}{\eta^{1+\bar\alpha}} \left(1 + \frac{C_1}{1-\eta^{\bar\alpha}} \right)\eta^{(j+1) (1+\bar\alpha)} \\
	&\leq\frac{1}{\eta^{1+\bar\alpha}} \left(1 + \frac{C_1}{1-\eta^{\bar\alpha}} \right)(|x|+|t|^{\frac{1}{2-\gamma}})^{1+\bar\alpha} ,
\end{align*}
which implies that $\tilde{u} \in C^{1,\bar\alpha}(0,0)$.

Finally, reverting the reduction to $u$, we conclude that 
\begin{equation*}
		|u(x,t)-L(x)|\leq CN(|x|+|t|^{\frac{1}{2-\gamma}})^{1+\bar\alpha} \quad \text{for all }(x,t) \in Q_1^+
\end{equation*}
and 
\begin{equation*}
		|Du(0,0)| \leq |a_\infty| N + |DL_0| \leq CN.
\end{equation*}
\end{proof}

For any $\Omega \subset \joinrel \subset Q_1^+$, \eqref{eq:main} becomes a uniformly parabolic equation. In general, uniformly parabolic equations are expected to exhibit higher regularity than degenerate equations. However, without additional assumptions such as H\"older continuity of $f$, the interior regularity of strong solutions to \eqref{eq:main} in $\Omega$ is at most $C_{\textnormal{loc}}^{1,\alpha}$, we refer to \cite[Chapter VII]{Lie96} for the notion of a strong solution. The following lemma states that such $\alpha$  can be arbitrarily close to 1.
\begin{lemma} [Interior $C^{1, 1^-}$-estimates] \label{lem:int_C11-}
	Suppose that $f \in L^\infty(Q_1^+) \cap C(Q_1^+)$ and let $u\in W^{2,p}_{\textnormal{loc}}(Q_1^+) \cap L^p(Q_1^+)$ be a strong solution to \eqref{eq:main} in $Q_1^+$. Then $u \in C_{\textnormal{loc}}^{1,1-\frac{n}{p}}(Q_1^+)$ for all $p>n$ with the uniform estimate
\begin{equation*}
	\|u\|_{C^{1,1-\frac{n}{p}}(\overline{\Omega})} \leq C (\|u\|_{L^{\infty}(Q_1^+)} + \|f\|_{L^{\infty}(Q_1^+)}),
\end{equation*}
where $\Omega \subset \joinrel \subset Q_1^+$ and $C>0$ is a constant depending only on $n$, $\gamma$ and $\Omega$.
\end{lemma}

\begin{proof}
Since  \eqref{eq:main} is uniformly parabolic in $\Omega$, it follows from the interior $W^{2,p}$-estimate for uniformly parabolic equations that 
\begin{equation*}
	\|D^2 u\|_{L^p(\Omega)} +\|u_t\|_{L^p(\Omega)} \leq C (\|u\|_{L^{\infty}(Q_1^+)} + \|f\|_{L^{\infty}(Q_1^+)}) \quad \text{for all }p>1. 
\end{equation*}
Therefore, by Sobolev embedding theorem, we conclude that
\begin{equation*}
	\|u\|_{C^{1,1-\frac{n}{p}}(\overline{\Omega})} \leq C \|u\|_{W^{2,p}(\Omega)} \leq C (\|u\|_{L^{\infty}(Q_1^+)} + \|f\|_{L^{\infty}(Q_1^+)}) \quad \text{for all } p >n .
\end{equation*}
\end{proof}

We are now ready to prove the first main theorem, \Cref{thm:main}.
\begin{proof} [Proof of \Cref{thm:main}]
For each $(x_0,t_0) \in \overline{Q_{1/2}^+}$, it suffices to find a linear function $L$ satisfying 
\begin{equation*}
	|L(x_0)| + |DL| \leq CN
\end{equation*}
and
\begin{equation*}
	|u(x,t)-L(x)| \leq CN(|x-x_0| + \sqrt{|t-t_0|})^{1+\bar\alpha} \quad  \text{for all }(x,t) \in Q_1^+ \cap Q_1(x_0,t_0).
\end{equation*}
where 
\begin{equation*}
	N\coloneqq \|u\|_{L^{\infty}(Q_1^+)}+\|f\|_{L^{\infty}(Q_1^+)}+ \|g\|_{C^{1,\beta}(\overline{Q_1^+})}
\end{equation*}
and $C>0$ is a constant depending only on $n$, $\beta$ and $\gamma$.

If $(x_0,t_0) \in S_{1/2}$,  the polynomial we are looking for is given in \Cref{thm:point_c1a}.  Next, let us assume that $(x_0,t_0)$ is an interior point and let $\bar{x}_0= (x_0',0)$. Then, by \Cref{thm:point_c1a}, there exist a linear function $L_1$ such that
\begin{equation*}
	|L_1(\bar{x}_0)|+|DL_1| \leq CN
\end{equation*}
and 
\begin{align}
	|u(x,t)-L_1(x)| &\leq CN(|x-\bar{x}_0|+ |t-t_0|^{\frac{1}{2-\gamma}} )^{1+\bar\alpha}  \nonumber \\
	&\leq CN(|x-\bar{x}_0|+ \sqrt{|t-t_0|} )^{1+\bar\alpha}  \label{u-l1}
\end{align}
for all $(x,t) \in Q_1^+\cap Q_1^+(\bar{x}_0,t_0)$ , where $C>0$ is a constant depending only on $n$, $\beta$ and $\gamma$.

Let $r$ be the $n$-th compoent of $x_0$. Then $v(x,t)\coloneqq u(x,t)-L_1(x)$ is a solution to the following uniformly parabolic equation: 
\begin{equation*}
	v_t - x_n^\gamma \Delta v = f \quad \text{in } \mathcal{Q}_{3r/4}(x_0,t_0).
\end{equation*}
 By \Cref{lem:int_C11-}, there exists a linear function $L_2$ such that 
\begin{equation} \label{norm_rl2}
	r |DL_2|  \leq C( \|v\|_{L^{\infty}(\mathcal{Q}_{3r/4}(x_0,t_0))}  + r^2 \|f\|_{L^\infty(Q_1^+)})
\end{equation}
and
\begin{equation} \label{v-ltwo} 
	|v(x,t)-L_2(x)|  \leq C \bigg( \frac{ \|v\|_{L^{\infty}(\mathcal{Q}_{4r/3}(x_0,t_0))}}{r^{1+\bar\alpha}} + r^{1-\bar\alpha} \|f\|_{L^\infty(Q_1^+)}  \bigg) (|x-x_0| + \sqrt{|t-t_0|})^{1+\bar\alpha} 
\end{equation}
for all $(x,t) \in \mathcal{Q}_{r/2} (x_0,t_0)$, where $C>0$ is a constant depending only on $n$ and $\gamma$.
From \eqref{u-l1}, we obtain
\begin{align} 
	|v(x,t)| &\leq CN (|x-\bar{x}_0|+\sqrt{|t-t_0|} )^{1+\bar\alpha} \nonumber \\
	&\leq CN(|x-x_0|+\sqrt{|t-t_0|} +r \big)^{1+\bar\alpha} \nonumber \\
	& \leq CN r^{1+\bar\alpha} \quad \text{for all }(x,t) \in \mathcal{Q}_{3r/4}(x_0,t_0). \label{bdd_v}
\end{align}
From \eqref{norm_rl2} and \eqref{bdd_v}, we have  
\begin{equation} \label{norm_l2}  	
	|L_2(x_0)|+ r |DL_2|  \leq CN r^{1+\bar\alpha}.
\end{equation}
Moreover, combining \eqref{v-ltwo} and \eqref{bdd_v} gives  
\begin{equation} \label{final1}
	|u(x,t)-L_1(x)-L_2(x)| \leq CN (|x-x_0| + \sqrt{|t-t_0|})^{1+\bar\alpha} \quad \text{for all } (x,t) \in  \mathcal{Q}_{r/2}(x_0,t_0)
\end{equation}
and combining \eqref{u-l1} and \eqref{norm_l2} gives
\begin{align}
	|u(x,t)-L_1(x)-L_2(x)| &\leq |u(x,t)-L_1(x)| + |L_2(x)| \nonumber\\
	&\leq CN(|x-\bar{x}_0|+ \sqrt{|t-t_0|} )^{1+\bar\alpha} +|L_2(x_0)|  +|DL_2| |x-x_0|  \nonumber\\
	&\leq CN(|x-\bar{x}_0|+ \sqrt{|t-t_0|} )^{1+\bar\alpha}  + CNr^{1+\bar\alpha}+ CNr^{\bar\alpha} |x-x_0| \nonumber\\
	&\leq CN(|x-x_0|+ \sqrt{|t-t_0|} )^{1+\bar\alpha} \label{final2}
\end{align}
for all $(x,t) \in( Q_1^+\cap Q_1^+(x_0,t_0)\cap Q_1^+(\bar{x}_0,t_0))\setminus \mathcal{Q}_{r/2}(x_0,t_0)$. Since $|x-x_0|+ \sqrt{|t-t_0|} > r/2$ in \eqref{final2}, the estimate \eqref{final2} can be extended up to $(Q_1^+\cap Q_1^+(x_0,t_0))\setminus \mathcal{Q}_{r/2}(x_0,t_0)$ by enlarging the constant $C$.

Therefore, letting $L = L_1+L_2$ and combining \eqref{final1} and \eqref{final2}, we obtain the desired conclusion.
\end{proof}

\subsection{Optimal regularity for homogeneous equations}  \label{sec:OR_hom}
In this subsection, we address the optimal regularity of solutions to \eqref{prob_xf}. The $C^{2,\alpha}$-regularity of solutions to the homogeneous equation played a crucial role in establishing the optimal $C^{1,1-\gamma}$-regularity of solutions to the inhomogeneous equation. However, in the case of the homogeneous equation, this approach is no longer applicable. We address this issue through time derivative estimates.
\begin{lemma} \label{lem:hol_ut/x_g}
Suppose  $g \in C^{2,\beta}(\overline{S_1})$ satisfying \eqref{cond_gt} for some $\beta  \in (0,1)$. Let $u$ be a viscosity solution to \eqref{prob_xf}. Then $x_n^{-\gamma} u_t \in C^{1-\gamma}(\overline{Q_{1/2}^+})$ with the uniform estimate
\begin{equation*}
	\|x_n^{-\gamma}u_t\|_{C^{1-\gamma}(\overline{Q_{1/2}^+})} \leq C( \|u\|_{L^{\infty}(Q_1^+)} + \|g\|_{C^{2,\beta}(\overline{S_1})}) ,
\end{equation*}
where $C>0$ is a constant depending only on $n$, $\beta$ and $\gamma$.
\end{lemma}

\begin{proof}
Let $(x,t), (y,s) \in Q_{1/2}^+$. Without loss of generality, we may assume that $x_n \leq y_n\eqqcolon r$.  Since the equation in \eqref{prob_xf} is uniformly parabolic in any $\Omega \subset \joinrel \subset Q_1^+$, $u \in C_{\textnormal{loc}}^{\infty}(Q_1^+)$. Moreover, it follows from \eqref{cond_gt} that $u_t$ is a solution to \eqref{eq:f=g=0}.
Then, combining \Cref{rmk_ut_bdd}, \Cref{thm:LY25_lip} and \cite[Theorem 1.5 and Theorem 1.7]{LY25} gives
\begin{equation} \label{time_lip}
\begin{aligned}
	|u_t(x,t)|  &\leq C \|u_t\|_{L^\infty(Q_{3/4}^+)}  x_n \\
	&\leq C \|u\|_{C^2(\overline{Q_{3/4}^+})}  x_n \\
	&\leq C(\|u\|_{L^\infty(Q_1^+)} +\|g\|_{C^{2,\beta}(\overline{S_1})}) x_n
\end{aligned}
\end{equation}
for all $(x,t) \in \overline{Q_{1/2}^+}$. This implies that 
\begin{equation} \label{p_holder}
	|x_n^{-\gamma}u_t(x,t)| \leq   C(\|u\|_{L^\infty(Q_1^+)} + \|g\|_{C^{2,\beta}(\overline{S_1})})  (|x|+|t|^{\frac{1}{2-\gamma}})^{1-\gamma}   \quad \text{for all } (x,t) \in \overline{Q_{1/2}^+}.
\end{equation}

\noindent {\textbf{Case 1: $|x-y|+|t-s|^{\frac{1}{2-\gamma}}\leq r/2.$}}

Let $v(x,t)=x_n^{-\gamma}u_t(x,t)$. From \eqref{p_holder}, we have 
\begin{equation*}
	 \|v\|_{L^\infty(Q_r(y,s))} \leq C(\|u\|_{L^\infty(Q_1^+)} + \|g\|_{C^{2,\beta}(\overline{S_1})}) r^{1-\gamma}.
\end{equation*}
Furthermore, by interior H\"older estimates for uniformly parabolic equations, we have 
\begin{equation} \label{lip1}
\begin{aligned}
	|x_n^{-\gamma}u_t(x,t) - y_n^{-\gamma}u_t(y,s)| &\leq \frac{C \|v\|_{L^\infty(Q_r(y,s))}}{r^{1-\gamma}} (|x-y|+\sqrt{|t-s|})^{1-\gamma} \\
	&\leq C(\|u\|_{L^\infty(Q_1^+)} + \|g\|_{C^{2,\beta}(\overline{S_1})}) (|x-y|+\sqrt{|t-s|})^{1-\gamma}.
\end{aligned}
\end{equation}

\noindent {\textbf{Case 2: $|x-y|+|t-s|^{\frac{1}{2-\gamma}} > r/2.$}}

From \eqref{time_lip}, we obtain
\begin{align}
	|x_n^{-\gamma}u_t(x,t)-y_n^{-\gamma}u_t(y,s)|&\leq |x_n^{-\gamma}u_t(x,t)| + |y_n^{-\gamma}u_t(y,s)| \nonumber \\
	&\leq C (\|u\|_{L^\infty(Q_1^+)} +\|g\|_{C^{2,\beta}(\overline{S_1})}) (x_n^{1-\gamma}+y_n^{1-\gamma}) \nonumber\\
	&\leq C(\|u\|_{L^\infty(Q_1^+)} + \|g\|_{C^{2,\beta}(\overline{S_1})})  r^{1-\gamma} \nonumber\\
	&\leq C(\|u\|_{L^\infty(Q_1^+)} + \|g\|_{C^{2,\beta}(\overline{S_1})}) (|x-y|+|t-s|^{\frac{1}{2-\gamma}})^{1-\gamma}\nonumber \\
	&\leq C(\|u\|_{L^\infty(Q_1^+)} + \|g\|_{C^{2,\beta}(\overline{S_1})}) (|x-y|+\sqrt{|t-s|})^{1-\gamma} .  \label{lip2}
\end{align}
Combining \eqref{lip1} and \eqref{lip2}, we obtain the conclusion. 
\end{proof}

Time derivative estimates \Cref{lem:hol_ut/x_g} transforms the homogeneous parabolic equation into an elliptic equation with a H\"older source term on a fixed time slice, which is the key idea in establishing the optimal regularity of solutions to \eqref{prob_xf}. Based on this idea, we present the proof of \Cref{thm:main2}.
\begin{proof} [Proof of \Cref{thm:main2}]
For fixed $t \in [-1/4,0]$, $v(x) \coloneqq u(x,t)$ is a solution to the following elliptic problem:
\begin{equation} \label{eq:elliptic}
\left\{\begin{aligned}
	\Delta v&= \tilde{f} && \text{in } B_1^+ \\
	v&=\tilde{g} && \text{on } S_1
\end{aligned}\right.
\end{equation}
where $\tilde{f} \coloneqq x_n^{-\gamma} u_t  (\cdot, t)$ and $\tilde{g} \coloneqq g (\cdot, t)$. By \Cref{lem:hol_ut/x_g}, we know that $\tilde{f}  \in  C^{1-\gamma}(\overline{B_{3/4}^+})$ and 
\begin{equation} \label{tilde_f_hol}
	\|\tilde{f}\|_{C^{1-\gamma}(\overline{B_{3/4}^+})} \leq C( \|u\|_{L^{\infty}(Q_1^+)} + \|g\|_{C^{2,\beta}(\overline{S_1})}).
\end{equation}
Thus, by \eqref{tilde_f_hol} and the Schauder estimates for elliptic equations, we have $v\in C^{2,\bar\alpha}(\overline{B_{1/2}^+})$ with the uniform estimate
\begin{align*}  
	\|u(\cdot,t)\|_{C^{2,\bar\alpha}(\overline{B_{1/2}^+})} &= \|v\|_{C^{2,\bar\alpha}(\overline{B_{1/2}^+})} \\
	&\leq C(\|v\|_{L^{\infty}(Q_1^+)}  + \|\tilde{f} \|_{C^{1-\gamma}(\overline{B_{3/4}^+})} + \|\tilde{g}\|_{C^{2,\beta}(\overline{S_1})} ) \\
	&\leq C (\|u\|_{L^{\infty}(Q_1^+)} + \|g\|_{C^{2,\beta}(\overline{S_1})}).
\end{align*}
\end{proof}

\Cref{cor:main2} follows from \Cref{thm:main2} and the following lemma.
\begin{lemma} \label{lem:hol_ut/x}
Let $\kappa \in\mathbb{N}_0^n$ and let $u$ be a viscosity solution to \eqref{eq:f=g=0}. Then $x_n^{-\gamma} D_{(x',t)}^\kappa u \in C^{1-\gamma}(\overline{Q_{1/2}^+})$ with the uniform estimate
\begin{equation*}
	\|x_n^{-\gamma} D_{(x',t)}^\kappa u\|_{C^{1-\gamma}(\overline{Q_{1/2}^+})} \leq C \|u\|_{L^{\infty}(Q_1^+)},
\end{equation*}
where $C>0$ is a constant depending only on $n$ and $\gamma$.
\end{lemma}

\begin{proof}
Since the equation in \eqref{eq:f=g=0} is uniformly parabolic in any $\Omega \subset \joinrel \subset Q_1^+$, $u \in C_{\textnormal{loc}}^{\infty}(Q_1^+)$. We now consider the function 
\begin{equation*}
	w(x,t) = D_{(x',t)}^\kappa u(x,t) \quad \text{for } (x,t) \in Q_1^+ \text{ and } \kappa \in \mathbb{N}_0^n.
\end{equation*}

The equation in \eqref{eq:f=g=0} is homogeneous, and since $u=0$ on $S_1$,  differentiating both sides of the equation in \eqref{eq:f=g=0}  with respect to the variables $x'$ and $t$ yields that the function $w$ is again a solution to \eqref{eq:f=g=0}.

Furthermore, it follows from the bootstrap argument for the estimates in \Cref{thm:LY25_c2a} that
\begin{equation} \label{bdd_w}
	\|w\|_{L^\infty(Q_{3/4}^+)} \leq C\|u\|_{L^\infty(Q_1^+)} .
\end{equation}
The remaining part of the proof is exactly the same as that of \Cref{lem:hol_ut/x_g} if we consider $x_n^{-\gamma}w$ and $g=0$ instead of $x_n^{-\gamma}u_t$ and $g$.
\end{proof}

\begin{proof} [Proof of \Cref{cor:main2}]
For fixed $t \in [-1/4,0]$ and $\kappa \in \mathbb{N}_0^n$, $w(x) \coloneqq D_{(x',t)}^\kappa u(x,t)$ is a solution to the following elliptic problem:
\begin{equation} \label{eq:elliptic}
\left\{\begin{aligned}
	\Delta w&= \tilde{f} && \text{in } B_{3/4}^+ \\
	w&= 0 && \text{on }S_1,
\end{aligned}\right.
\end{equation}
where $\tilde{f} \coloneqq  x_n^{-\gamma} D_{(x',t)}^{\kappa+e_n} u (\cdot, t)$.
It follows from \Cref{lem:hol_ut/x} that $\tilde{f} \in  C^{1-\gamma}(\overline{B_{3/4}^+})$ and 
\begin{equation} \label{hol_f_1-g}
	\|\tilde{f} \|_{C^{1-\gamma}(\overline{B_{3/4}^+})} \leq C \|u\|_{L^{\infty}(Q_1^+)},
\end{equation}
Thus, by \eqref{bdd_w},  \eqref{hol_f_1-g}, and the Schauder estimates for elliptic equations,  we have $v \in C^{2,1-\gamma}(\overline{B_{1/2}^+})$ with the uniform estimate
\begin{align*}  
	\|D_{(x',t)}^\kappa u(\cdot,t)\|_{C^{2,1-\gamma}(\overline{B_{1/2}^+})} & = \|w\|_{C^{2,1-\gamma}(\overline{B_{1/2}^+})}  \\
	&\leq C(\|w\|_{L^{\infty}(Q_{3/4}^+)}  + \|\tilde{f} \|_{C^{1-\gamma}(\overline{B_{3/4}^+})} ) \\
	&\leq C \|u\|_{L^{\infty}(Q_1^+)}.
\end{align*}
\end{proof}


%
%


\end{document}